\documentclass[10pt]{amsart}

\usepackage{enumerate}
\usepackage{amsmath,amssymb}

\usepackage{combelow}

\newtheorem{theorem}{Theorem}
\newtheorem{lemma}[theorem]{Lemma}
\newtheorem{proposition}[theorem]{Proposition}

\def\cM{{\mathcal{M}}}
\def\cP{{\mathcal{P}}}

\begin{document}
\title[Distance from a projection to nilpotents]{Determination of the distance from a projection to nilpotents}

\author[M. Izumi]{Masaki Izumi}
\address{Graduate School of Science,
Kyoto University,
Sakyo-ku, Kyoto 606-8502, Japan.}
\email{izumi@math.kyoto-u.ac.jp}

\author[M. Mori]{Michiya Mori}
\address{Graduate School of Mathematical Sciences, The University of Tokyo, 3-8-1 Komaba, Meguro-ku, Tokyo, 153-8914, Japan; Interdisciplinary Theoretical and Mathematical Sciences Program (iTHEMS), RIKEN, 2-1 Hirosawa, Wako, Saitama 351-0198, Japan.}
\email{mmori@ms.u-tokyo.ac.jp}

\thanks{The first author was supported in part by JSPS KAKENHI Grant Number 20H01805.
The second author was supported in part by JSPS KAKENHI Grant Number 22K13934.}
%\thanks{}
\subjclass[2020]{Primary 15A60, 47A30.} 

\keywords{projection, nilpotent, distance}

\date{}

\begin{abstract}
In this note, we study the distance from an arbitrary nonzero projection $P$ to the set of nilpotents in a factor $\mathcal{M}$ equipped with a normal faithful tracial state $\tau$. 
We prove that the distance equals $(2\cos \frac{\tau(P)\pi}{1+2\tau(P)})^{-1}$.
This is new even in the case where $\mathcal{M}$ is the matrix algebra. 
The special case settles a conjecture posed by Z. Cramer. 
\end{abstract}

\maketitle

\section{Introduction}
Let $\cM$ be a factor (a von Neumann algebra with trivial center) acting on a separable Hilbert space.
A \emph{nilpotent} is an operator $X\in \cM$ satisfying $X^n=0$ for some positive integer $n$.  
For $X\in \cM$, let $\nu_\cM(X)$ denote the distance from $X$ to the set of nilpotents in $\cM$.
In other words, 
\[
\nu_\cM(X)=\inf\{\lVert X-N\rVert\mid N \text{ is a nilpotent in }\cM\},
\]
where $\lVert \cdot\rVert$ denotes the operator norm.

In 1970, Paul Halmos proposed a list of 10 problems in operator theory \cite{Ha}. 
Problem 7 in the list asks the following. 
Let $\mathbb{B}(\mathcal{H})$ denote the von Neumann algebra (type I$_\infty$ factor) of all bounded linear operators on a separable infinite-dimensional complex Hilbert space $\mathcal{H}$. 
Let $X\in \mathbb{B}(\mathcal{H})$ be a \emph{quasinilpotent} operator, that is, an operator whose spectrum $\sigma(X)$ consists of only one point $0$. Is it true that $\nu_{\mathbb{B}(\mathcal{H})} (X)=0$?

The answer to this problem is affirmative. 
In fact, the theorem below is given by Apostol, Foia\cb{s}, Voiculescu in \cite{AFV}, which completely characterizes the condition for an operator $X$ to satisfy $\nu_{\mathbb{B}(\mathcal{H})} (X)=0$.
\begin{theorem}
Let $X\in \mathbb{B}(\mathcal{H})$. 
The equation $\nu_{\mathbb{B}(\mathcal{H})} (X)=0$ holds if and only if the following three conditions hold. 
\begin{itemize}
\item The spectrum $\sigma(X)$ of $X$ is connected and $0\in \sigma(X)$. 
\item The essential spectrum $\sigma_e(X)$ of $X$ is connected and $0\in \sigma_e(X)$.
\item If $\lambda\in \mathbb{C}$ and $\lambda I-X$ is semi-Fredholm, then the Fredholm index of $\lambda I-X$ equals $0$.
\end{itemize}
\end{theorem}

It is natural to get interested in the value $\nu_\cM(X)$ for a more general factor $\cM$ and an operator $X\in \cM$.
In this note, we completely compute $\nu_\cM(P)$ for a projection $P$ in a general factor $\cM$. 
Recall that a \emph{projection} is an operator $P\in \cM$ satisfying $P=P^2=P^*$.
Let $\cP(\cM)$ denote the set of projections in $\cM$.
For $P\in \cP(\cM)$, we use the symbol $P^\perp := I-P$.
It is clear that $\nu_\cM(0)=0$.

It appears that the oldest result on $\nu_\cM(P)$ for a projection $P$ is given by Hedlund in 1972 \cite{Hed}. 
Herrero {\cite[Corollary 9]{Her1}} completely computed $\nu_\cM(P)$ when $\cM$ is a type I$_\infty$ factor. 
From this result, it is not difficult to get the following. 
\begin{theorem}[{\cite[Corollary 9]{Her1}}, see also {\cite[Section 8]{S1}}]
Let $\cM$ be an infinite factor and $P\in \cP(\cM)\setminus\{0\}$. 
\begin{itemize}
\item If $P^\perp\in \cP(\cM)$ is an infinite projection, then $\nu_\cM(P)=1/2$. 
\item If $P^\perp\in \cP(\cM)$ is a finite projection, then $\nu_\cM(P)=1$. 
\end{itemize}
\end{theorem}

Therefore, we restrict our attention to finite factors. 
In the case of matrix algebras, one may find some research in the literature. 
Let $n$ be a positive integer and let $\mathbb{M}_n$ denote the algebra of $n\times n$ complex matrices, which is a type I$_n$ factor.
MacDonald {\cite[Theorem 1]{M1}} proved  that $\nu_{\mathbb{M}_n}(P)=(2\cos\frac{\pi}{n+2})^{-1}$ for every projection $P\in \cP(\mathbb{M}_n)$ of rank $1$.
Cramer {\cite[Theorem 3.6]{C}} proved that $\nu_{\mathbb{M}_n}(P)=(2\cos\frac{\pi}{\frac{n}{n-1}+2})^{-1}$ for every projection $P\in \cP(\mathbb{M}_n)$ of rank $n-1$. 
He conjectured that $\nu_{\mathbb{M}_n}(P)=(2\cos\frac{\pi}{\frac{n}{m}+2})^{-1}$ for every projection $P\in \cP(\mathbb{M}_n)$ of rank $m$, $1\leq m\leq n$ \cite[Conjecture 5.1]{C}. 
The second author recently proved that $\nu_{\mathbb{M}_n}(P)\geq(2\cos\frac{\pi}{n+2})^{-1}$ for every nonzero projection $P\in \cP(\mathbb{M}_n)$ \cite{Mo}, settling a conjecture posed by MacDonald in \cite{M1}. 
For more information concerning related topics, see \cite[Chapter 2]{Her5}, \cite{C}, \cite{M2}.
See also \cite{S1}, \cite{S2}, where the distance to the set of (quasi)nilpotent operators is studied in various settings of C$^*$- or von Neumann algebras.

In what follows, let $\cM$ be a finite factor. 
Then $\cM$ is equipped with a unique normal faithful tracial state $\tau$. 
As one may see from the preceding paragraph, the computation of the exact value of $\nu_\cM(P)$ was not complete even in the case of matrix algebras, and little was known about the case of type II$_1$ factors. 
The goal of this note is to prove
\begin{theorem}\label{tau}
If $0\neq P\in \cP(\cM)$, then $\nu_\cM(P) = (2\cos \theta)^{-1}$, where $\theta =\displaystyle \frac{\tau(P)\pi}{1+2\tau(P)}$.
\end{theorem}
This completes the program of the determination of $\nu_\cM(P)$ for a general factor $\cM$ and a general projection $P\in \cP(\cM)$.
Note that $\tau(X)=\mathrm{tr}\, X/n$ when $\cM=\mathbb{M}_n$ and $X\in \cM$.
It follows that Theorem \ref{tau} applied to a type I$_n$ factor settles Cramer's conjecture.

%We remark that the techniques employed in this note together with the use of center-valued trace enables us to compute the distance from a projection to the set of nilpotents in the setting of general von Neumann algebras (possibly acting on a nonseparable Hilbert space), too. We chose to study only the case of factors to simplify the arguments.

\section{Results}
For $X\in \cM$, the projection $E$ onto the closure of the range of $X$ belongs to $\cP(M)$. 
We call $E$ the range projection of $X$.
\begin{lemma}
Let $n\geq 1$. 
An operator $N\in \cM$ satisfies $N^n=0$ if and only if there are projections $0=P_0\leq P_1\leq \cdots\leq P_n=I$ in $\cM$ satisfying $P_{k-1}^\perp NP_k=0$ for every $1\leq k\leq n$.
\end{lemma}
\begin{proof}
If $N^n=0$, then define $P_k$ to be the range projection of $N^{n-k}$.
The converse is easier.
\end{proof}

\begin{lemma}\label{l0}
Let $0=P_0\leq P_1\leq \cdots\leq P_n=I$ be in $\cP(\cM)$ and $X\in \cM$.
Set $\mathcal{N}:=\{N\in \cM\mid P_{k-1}^\perp NP_k=0\text{ for every }1\leq k\leq n\}$. 
Then there is a nilpotent $N_0\in \mathcal{N}$ satisfying 
\[
\lVert X-N_0\rVert=\inf_{N\in \mathcal{N}} \lVert X-N\rVert =\max_{1\leq k\leq n}\lVert P_{k-1}^\perp XP_k\rVert.
\]
\end{lemma}
\begin{proof}
This is a version of the so-called Arveson's distance formula \cite{A}. See \cite[Theorem 1]{Pa} and \cite{P}. 
\end{proof}

For $X\in \cM$, let $\delta(X)$ denote the trace of the range projection of $X$.
\begin{lemma}
If $X, Y\in \cM$, then $\delta(XY), \delta(YX)\leq \delta(X)$ and $\delta(X+Y)\leq \delta(X)+\delta(Y)$ hold.
\end{lemma}
\begin{proof}
It is clear that $\delta(XY)\leq \delta(X)$. 
If $Z=V\lvert Z\rvert$ is the polar decomposition of an operator $Z\in \cM$,
then $VV^*$ is the range projection of $Z$ and $V^*V$ is the range projection of $Z^*$, hence $\delta(Z)=\tau(VV^*)=\tau(V^*V)=\delta(Z^*)$. 
Thus we have $\delta(YX)=\delta((YX)^*)=\delta(X^*Y^*)\leq \delta(X^*)=\delta(X)$.

Let $E_1,E_2, E_3$ denote the range projections of $X+Y, X, Y$, respectively. 
Then we have $E_1\leq E_2\vee E_3$. Since $E_2\vee E_3-E_2$ is Murray--von Neumann equivalent to $E_3-E_2\wedge E_3$, we have $\tau(E_2\vee E_3-E_2)=\tau(E_3-E_2\wedge E_3)$, which implies $\tau(E_2\vee E_3)\leq \tau(E_2)+\tau(E_3)$. Thus $\delta(X+Y)\leq \delta(X)+\delta(Y)$ holds.
\end{proof}

\begin{lemma}\label{l1}
Let $X\in \cM$ satisfy $X\geq 0$. Let $n\geq 1$ and let $0=P_0\leq P_1\leq \cdots\leq P_n=I$ be in $\cP(\cM)$.
For each $k$, put $A_k=X^{1/2}P_k X^{1/2}$.
Then
\[
\lVert P_{k-1}^\perp XP_k\rVert = \lVert (X-A_{k-1})^{1/2}\cdot A_k^{1/2}\rVert.
\]
Moreover, 
\begin{equation}\label{Ak}
0=A_0\leq A_1\leq \cdots\leq A_n=X\text{ and } \sum_{k=1}^n\delta(A_k-A_{k-1})\leq 1
\end{equation}
hold.
\end{lemma}
\begin{proof}
We have
\[
\begin{split}
\lVert P_{k-1}^\perp XP_k\rVert &= \lVert P_{k-1}^\perp X^{1/2}\cdot X^{1/2}P_k\rVert\\
&= \lVert (X^{1/2}P_{k-1}^\perp X^{1/2})^{1/2}\cdot (X^{1/2}P_kX^{1/2})^{1/2}\rVert\\
&= \lVert (X-X^{1/2}P_{k-1} X^{1/2})^{1/2}\cdot (X^{1/2}P_kX^{1/2})^{1/2}\rVert\\
&= \lVert (X-A_{k-1})^{1/2}\cdot A_k^{1/2}\rVert.
\end{split}
\]
It is clear that $0=A_0\leq A_1\leq \cdots\leq A_n=X$. 
Since $\delta(A_k-A_{k-1})\leq \delta(P_k-P_{k-1})=\tau(P_k-P_{k-1})$, we get $\sum_{k=1}^n\delta(A_k-A_{k-1})\leq 1$.
\end{proof}

Conversely, the following holds.
\begin{lemma}\label{l2}
Let $X\in \cM$ satisfy $X\geq 0$.
Let $n\geq 1$ and $A_0, A_1,\dots, A_n\in \cM$ satisfy \eqref{Ak}. 
Then there are projections $0=P_0\leq P_1\leq \cdots\leq P_n=I$ in $\cM$ satisfying $X^{1/2}P_kX^{1/2}=A_k$ for every $k$.
\end{lemma}
\begin{proof}
Since $\displaystyle \sum_{k=1}^n\delta(A_k-A_{k-1})\leq 1$, one may choose operators $Y_1, Y_2, \ldots, Y_n\in \cM$, whose range projections $Q_1, Q_2, \ldots, Q_n$ are pairwise orthogonal, with the property $Y_k^*Y_k= A_k-A_{k-1}$.
Then $Y:=Y_1+Y_2+\cdots +Y_n$ satisfies $Y^*(Q_1+\cdots+Q_k)Y=A_k$ for every $k$. 
Since $\cM$ is of finite type and $Y^*Y=X$, one may take a unitary operator $U\in \cM$ with $Y=UX^{1/2}$.
Thus $P_k=U^*(Q_1+\cdots Q_k)U$, $k\leq n-1$,  and $P_n=I$ satisfy the desired property.
\end{proof} 

Therefore, we get the following. 
\begin{proposition}\label{equivalent}
Let $X\in \cM$ and $X\geq 0$. 
Then
\[
\nu_\cM(X)= \inf \max_{1\leq k\leq n} \lVert (X-A_{k-1})^{1/2}\cdot A_k^{1/2}\rVert,
\]
where the infimum in the right-hand side is taken over the collection of all $n\geq 1$ and $(A_0, A_1,\dots, A_n)\in \cM^{n+1}$ satisfying \eqref{Ak}.
\end{proposition}

In what follows, we study the case $0\neq X=P\in \cP(\cM)$. 
Let $n\geq 1$ and $(A_0, A_1,\dots, A_n)\in \cM^{n+1}$ satisfy \eqref{Ak}.
Set $\displaystyle \theta := \frac{\tau(P)\pi}{1+2\tau(P)}$ and $\varphi:=\theta/2$.
To get Theorem \ref{tau}, we need to compare $(2\cos\theta)^{-1}$ with
\[
\max_{1\leq k\leq n}\lVert (P-A_{k-1})^{1/2}A_k^{1/2}\rVert=
\max_{1\leq k\leq n}\lVert (I-A_{k-1})^{1/2}A_k^{1/2}\rVert.
\]
Observe that the mapping 
\[
t\mapsto \frac{1}{2\cos\theta} \cdot \frac{\sin t}{\sin (t +\theta)}
\]
is a monotone increasing bijection from $[0, \pi-2\theta]$ onto $[0,1]$.
Since $0\leq A_k\leq P$, by functional calculus we may uniquely take an operator $B_k\in \cM$ satisfying $0\leq B_k\leq (\pi-2\theta) P$ and 
\begin{equation}\label{akbk}
A_k= \frac{1}{2\cos\theta} \cdot \frac{\sin B_k}{\sin (B_k +\theta I)}.
\end{equation}
(Note that the right-hand side makes sense because $\sin(B_k +\theta I)$ is invertible and the pair $\sin B_k$, $\sin(B_k +\theta I)$ commutes. In what follows we frequently use such fractional expressions.)
Since $A_0\leq A_1\leq \cdots\leq A_n$, we have 
\[
\frac{\sin B_0}{\sin (B_0 +\theta I)}\leq \frac{\sin B_1}{\sin (B_1 +\theta I)}\leq\cdots \leq \frac{\sin B_n}{\sin (B_n +\theta I)}.
\]

\begin{lemma}\label{la}
Let $B, {B'}\in \cM$ satisfy $0\leq B\leq (\pi-2\theta) I$, $0\leq {B'}\leq (\pi-2\theta) I$. 
Then 
%\begin{itemize}
$\displaystyle \frac{\sin B}{\sin (B +\theta I)}\leq \frac{\sin {B'}}{\sin ({B'} +\theta I)}$ holds
if and only if
$\cot (B+{\varphi}I)\geq \cot ({B'}+{\varphi}I)$.
Moreover, we have
\[
\delta\left(\frac{\sin B}{\sin (B +\theta I)}-\frac{\sin {B'}}{\sin ({B'} +\theta I)}\right) = \delta(\cot (B+{\varphi}I)- \cot ({B'}+{\varphi}I)).
\]
%\end{itemize}
\end{lemma}
\begin{proof}
By applying the formula $\sin(\alpha \pm\beta)=\sin \alpha \cos\beta\pm \cos \alpha\sin \beta$ to mutually commuting operators, we get
\[
\begin{split}
\sin(B+{\varphi}I\pm {\varphi}I) &= \sin (B+{\varphi}I) \cos({\varphi}I)\pm \cos (B+{\varphi}I) \sin ({\varphi}I)\\
&= \cos{\varphi} \sin (B+{\varphi}I)\pm \sin {\varphi}\cos (B+{\varphi}I). 
\end{split}
\]
Thus 
\[
\frac{\sin B}{\sin (B +\theta I)} = \frac{\cos{\varphi} \sin (B+{\varphi}I)- \sin {\varphi}\cos (B+{\varphi}I)}{\cos{\varphi} \sin (B+{\varphi}I)+ \sin {\varphi}\cos (B+{\varphi}I)}
= \frac{2I}{I+ \tan {\varphi}\cot (B+{\varphi}I)}-I.
\]
We also get 
\[
\frac{\sin {B'}}{\sin ({B'} +\theta I)} = 
\frac{2I}{I+ \tan {\varphi}\cot ({B'}+{\varphi}I)}-I.
\]
Therefore, 
$\displaystyle \frac{\sin B}{\sin (B +\theta I)}\leq \frac{\sin {B'}}{\sin ({B'} +\theta I)}$ holds if and only if 
\[
\frac{2I}{I+ \tan {\varphi}\cot (B+{\varphi}I)}\leq \frac{2I}{I+ \tan {\varphi}\cot ({B'}+{\varphi}I)}. 
\]
Since $X\leq Y\iff X^{-1}\geq Y^{-1}$ for positive invertible operators, we see that the above condition is equivalent to $\cot (B+{\varphi}I)\geq \cot ({B'}+{\varphi}I)$.

Using the identity $X^{-1}-Y^{-1} =X^{-1}(Y-X)Y^{-1}$ for invertible operators, we also get the equality concerning $\delta$.
\end{proof}

\begin{lemma}\label{XY}
Let $B, {B'}\in \cM$ satisfy $0\leq B\leq (\pi-2\theta) I$, $0\leq {B'}\leq (\pi-2\theta) I$. 
Put 
\[
A= \frac{1}{2\cos\theta} \cdot \frac{\sin B}{\sin (B +\theta I)},\quad 
A'= \frac{1}{2\cos\theta} \cdot \frac{\sin {B'}}{\sin ({B'} +\theta I)}.
\]
Then $\displaystyle \lVert (I-A)^{1/2}(A')^{1/2}\rVert\leq(2\cos\theta)^{-1}$
if and only if 
$\cot ({B'}+{\varphi}I)\geq \cot (B+{3\varphi}I)$.
\end{lemma}
\begin{proof}
Since 
\[
\sin B + \sin (B+2\theta I)= 2\sin (B+\theta I)\cos (\theta I) = 2\cos\theta \sin (B+\theta I),
\]
we get 
\[
I- \frac{1}{2\cos\theta}\cdot \frac{\sin B}{\sin (B +\theta I)} = \frac{1}{2\cos\theta}\cdot \frac{\sin (B+2\theta I)}{\sin (B +\theta I)}.
\]
It follows that
\[
\lVert (I-A)^{1/2}(A')^{1/2}\rVert\leq (2\cos\theta)^{-1} \iff 
\left\lVert \left(\frac{\sin (B +2\theta I)}{\sin (B +\theta I)}\right)^{1/2}
\left(\frac{\sin {B'}}{\sin ({B'} +\theta I)}\right)^{1/2} \right\rVert \leq 1.
\]
Using the argument as in the proof of the preceding lemma, we get 
\[
\frac{\sin (B+2\theta I)}{\sin (B +\theta I)} = 
\frac{I + \tan {\varphi}\cot (B+{3\varphi}I)}{I - \tan {\varphi}\cot (B+{3\varphi}I)}
%\frac{2\cos{\varphi}}{\cos{\varphi} - \sin {\varphi}\cot \left(B+{3\varphi}I\right)}-I
= \frac{I+S}{I-S}
\]
and 
\[
\frac{\sin {B'}}{\sin ({B'} +\theta I)} = 
\frac{I- \tan {\varphi}\cot ({B'}+{\varphi}I)}{I + \tan {\varphi}\cot ({B'}+{\varphi}I)}
%\frac{2\cos{\varphi}}{\cos{\varphi} + \sin {\varphi}\cot \left({B'}+{\varphi}I\right)} -I.
= \frac{I-T}{I+T},
\]
where $S= \tan{\varphi} \cot(B+{3\varphi}I)$ and $T= \tan{\varphi} \cot({B'}+{\varphi}I)$.
(Observe that $-I\leq S\leq I$ and $-I\leq T \leq I$ hold. Note that $I-S$ and $I+T$ are invertible, while $I+S$ and $I-T$ are not necessarily invertible.)
Hence
\[
\begin{split}
\lVert (I-A)^{1/2}(A')^{1/2}\rVert\leq (2\cos\theta)^{-1} &\iff 
\left\lVert \left(\frac{I+S}{I-S}\right)^{1/2}
\left(\frac{I-T}{I+T}\right)^{1/2} \right\rVert \leq 1\\
&\iff \left(\frac{I-T}{I+T}\right)^{1/2}\left(\frac{I+S}{I-S}\right)\left(\frac{I-T}{I+T}\right)^{1/2} \leq I\\
&\iff (I-T)^{1/2}\left(\frac{2I}{I-S}-I\right)(I-T)^{1/2} \leq I+T\\
&\iff 2(I-T)^{1/2}(I-S)^{-1}(I-T)^{1/2} \leq 2I\\
&\iff (I-T)^{1/2}(I-S)^{-1}(I-T)^{1/2} \leq I\\
&\iff \lVert (I-S)^{-1/2}(I-T)^{1/2}\rVert\leq 1\\
&\iff (I-S)^{-1/2}(I-T)(I-S)^{-1/2}\leq I\\
&\iff I-T\leq I-S\\
&\iff S\leq T\\
&\iff \cot (B+{3\varphi}I)\leq  \cot({B'}+{\varphi}I).
\end{split}
\]
\end{proof}

\begin{lemma}\label{XYZW}
Let $X, Y, Z\in \cM$ be self-adjoint operators satisfying $X\geq Y\geq Z$. 
If $X-Z$ is invertible, then there uniquely exists a self-adjoint operator $W\in \cM$ satisfying $X\geq Y\geq W\geq  Z$, $\delta(X-Y)=\delta(X-W)$, and $\delta(X-W)+\delta(W-Z)=1$.
\end{lemma}
\begin{proof}
One may assume $Z=0$ without loss of generality. 
Then $X\geq Y\geq 0$ and $X$ is invertible. 
By multiplying $X^{-1/2}$ from both sides of each operator, we may also assume $X=I$ without loss of generality.
Then we have $I\geq Y\geq 0$.
Assume that a self-adjoint operator $W\in \cM$ satisfies the desired properties. 
From $\delta(I-W)+\delta(W)=1$, we see that $W$ is a projection. 
This together with $I\geq Y\geq W\geq  0$, $\delta(I-Y)=\delta(I-W)$ shows that $W=\chi_{\{1\}}(Y)$. Conversely, it is easy to see that $W=\chi_{\{1\}}(Y)$ satisfies the desired properties.
\end{proof}

\begin{lemma}\label{snumber}
Let $X, Y\in \cM$ be self-adjoint operators satisfying $X\leq Y$. 
If $f\colon\sigma(X)\cup \sigma(Y)\to \mathbb{R}$ is a monotone decreasing function, then $\tau(f(X))\geq \tau(f(Y))$.
\end{lemma}
\begin{proof}
By considering the pair $X+\alpha I$, $Y+\alpha I$ instead of $X$, $Y$ for a suitable real number $\alpha$, one may assume $\sigma(X), \sigma(Y)\subset (0,\infty)$ without loss of generality.
One may also assume that the range of $f$ is contained in $(0,\infty)$ by adding some constant function.
For a real number $t\in (0,1]$ and a positive operator $Z\in \cM$, we define
\[
\mu_t(Z) =\inf\{s\geq 0\mid \tau(\chi_{(s,\infty)}(Z))\leq t\}.
\]
By \cite[Lemma 2.5 (iii)]{FK}, we have $\mu_t(X)\leq \mu_t(Y)$ for every $t\in (0,1]$. 
This together with \cite[Remark 3.3]{FK} and the assumption on $f$ implies 
\[
\tau(f(X))= \int_0^1 f(\mu_t(X))\, dt \geq  \int_0^1 f(\mu_t(Y))\, dt =\tau(f(Y)).
\]
\end{proof}

\begin{lemma}\label{UV}
Let $U, V\in \cM$ be unitaries,  and $r,s$ real numbers with $0\leq r\leq 1$, $0<s<2\pi$. 
If $\delta(U-V)=r$ and $\delta(U-e^{is}V)=1-r$, then there is a projection $Q\in \cP(\cM)$ satisfying $\tau(Q)=r$ and $U=Ve^{isQ}$.
\end{lemma}
\begin{proof}
It follows from $\delta(U-V)=r$ and $\delta(U-e^{is}V)=1-r$ that $\delta(V^*U-I)=r$ and $\delta(V^*U-e^{is}I)=1-r$.
Considering the spectral distribution of the unitary $V^*U$, we see that there is a projection $Q\in \cP(\cM)$ satisfying $\tau(Q)=r$ and $V^*U=e^{isQ}$.
\end{proof}

\begin{lemma}\label{lb}
Let $B, {B'}\in \cM$ satisfy $0\leq B\leq (\pi-2\theta) I$, $0\leq {B'}\leq (\pi-2\theta) I$, and 
\[
\cot (B+{\varphi}I)\geq \cot ({B'}+{\varphi}I)\geq \cot (B+{3\varphi}I).
\]
Set 
$r=\delta(\cot (B+{\varphi}I)- \cot ({B'}+{\varphi}I) )$.
Then $\tau({B'})-\tau(B) \leq r\theta$.
Moreover, 
$\tau({B'})-\tau(B) = r\theta$
holds if and only if $\delta(\cot ({B'}+{\varphi}I)- \cot (B+{3\varphi}I) ) = 1-r$.
\end{lemma}
\begin{proof}
Since $\cot (B+{\varphi}I)- \cot (B+{3\varphi}I)$ is invertible, by Lemma \ref{XYZW} we may find an operator $C\in \cM$ satisfying $\delta(\cot (B+{\varphi}I)- C )=r$, $\delta(C- \cot (B+{3\varphi}I) ) = 1-r$, and 
\[
\cot (B+{\varphi}I)\geq \cot ({B'}+{\varphi}I)\geq C\geq \cot (B+{3\varphi}I).
\]
One may take an operator $D\in \cM$ with
${\varphi} I\leq D\leq (\pi-{\varphi})I$ and $C= \cot D$.
Then 
\[
\cot (B+{\varphi}I)\geq \cot ({B'}+{\varphi}I)\geq \cot D\geq \cot (B+{3\varphi}I). 
\]
Since the inverse function of $\cot$ is monotone decreasing, Lemma \ref{snumber} implies that
\[
\tau(B+{\varphi}I)\leq \tau({B'}+{\varphi}I)\leq \tau(D)\leq \tau (B+{3\varphi}I). 
\]
Hence
\[
\tau({B'})-\tau(B) =  \tau({B'}+{\varphi}I)-  \tau(B+{\varphi}I) \leq \tau(D)-\tau(B+{\varphi}I).
\]
Moreover, we see that $\tau({B'}+{\varphi}I)=\tau(D)$ holds if and only if 
$\cot ({B'}+{\varphi}I)= \cot D$, which is in turn equivalent to $\delta(\cot ({B'}+{\varphi}I)- \cot (B+{3\varphi}I) ) = 1-r$.
Consequently, to get the desired conclusion, it suffices to show that 
$\tau(D) -\tau(B+{\varphi}I) = r\theta$.
By $\delta(\cot (B+{\varphi}I)- C )=r$ and the equality
\begin{equation}\label{cote}
\begin{split}
\cot (B+{\varphi}I)- C&=\cot (B+{\varphi}I)- \cot D\\
&= i\left(\frac{2I}{e^{2i(B+\varphi I)}-I}+I\right) -i\left(\frac{2I}{e^{2iD}-I}+I\right)\\
&= 2i((e^{2i(B+\varphi I)}-I)^{-1}-(e^{2iD}-I)^{-1})\\
&= 2i(e^{2i(B+\varphi I)}-I)^{-1}(e^{2iD}-e^{2i(B+\varphi I)})(e^{2iD}-I)^{-1},
\end{split}
\end{equation}
we get 
\begin{equation}\label{ichi}
\delta(e^{2iD}-e^{2i(B+\varphi I)}) =r.
\end{equation} 
Similarly, $\delta(C- \cot (B+{3\varphi}I) ) = 1-r$ implies 
\begin{equation}\label{ni}
\delta (e^{2iD}-e^{2i(B+3\varphi I)}) =1-r.
\end{equation}
The equations \eqref{ichi} and \eqref{ni} and Lemma \ref{UV} imply that there is a projection $Q\in \cP(\cM)$ satisfying $\tau(Q)=r$ and $e^{2iD} = e^{2i(B+\varphi I)} \cdot e^{4i\varphi Q} $.
Thus Lemma \ref{trace} below implies that $\tau(2D)=\tau(2(B+\varphi I)) + \tau(4\varphi Q)$, hence 
$\tau(D)-\tau(B+\varphi I) =2\tau(Q)\varphi = r\theta$.
\end{proof}

\begin{lemma}[{Lemma 3 in Section I.6.11 of \cite{D}}]\label{dix}
Let $\mathcal{U}\subset \mathbb{C}$ be an open set and $f\colon \mathcal{U}\to \mathbb{C}$ an analytic function. 
Let $g\colon [0,1]\to \mathcal{M}$ be a norm-differentiable function. 
Assume that %there is a Jordan curve $\Gamma$ in $\mathcal{U}$ such that 
the spectra $\sigma(g(t))$, $t\in [0,1]$, are contained in a common compact subset of $\mathcal{U}$. 
Then 
\[
\frac{d}{dt}\tau(f\circ g(t))=\tau(f'(g(t))g'(t))
\] holds for every $t\in [0,1]$.
\end{lemma}

\begin{lemma}\label{sigma}
Let $X, Y\in \cM$ be positive operators. 
Then 
\[
\sigma(e^{iX}e^{iY})\subset \{e^{i\alpha}\mid 0\leq \alpha\leq \lVert X\rVert+ \lVert Y\rVert\}.
\]
\end{lemma}
\begin{proof}
This is trivial when $\lVert X\rVert+ \lVert Y\rVert\geq 2\pi$ (note that $e^{iX}e^{iY}$ is unitary).
Assume that $\lVert X\rVert+ \lVert Y\rVert<\alpha<2\pi$. 
We need to show that $e^{iX}e^{iY} -e^{i\alpha}I$ is invertible.
Observe that the convex hull of $\sigma(e^{iX})$ and that of $\sigma(e^{i\alpha}e^{-iY})$ do not intersect. 
Let $\mathcal{H}$ be the Hilbert space on which $\cM$ acts.
Observe that $\langle Zh, h\rangle$ lies in the convex hull of $\sigma(Z)$ for every normal operator $Z\in \cM$ and every unit vector $h$.
It follows that there is a positive real number $\varepsilon>0$ satisfying $\lvert\langle (e^{iX}-e^{i\alpha}e^{-iY})h, h\rangle\rvert \geq \varepsilon$ for every unit vector $h\in \mathcal{H}$.
This implies that $\lVert (e^{iX}-e^{i\alpha}e^{-iY})h\rVert\geq \varepsilon$ and $\lVert (e^{iX}-e^{i\alpha}e^{-iY})^*h\rVert\geq \varepsilon$ for every unit vector $h\in \mathcal{H}$.
Hence $e^{iX}-e^{i\alpha}e^{-iY}$ is invertible, so $e^{iX}e^{iY} -e^{i\alpha}I$ is also invertible.
\end{proof}

\begin{lemma}[{See also Lemma 4 in Section I.6.11 of \cite{D}}]\label{trace}
Let $X, Y, Z\in \cM$ be positive operators satisfying $\lVert X\rVert +\lVert Y\rVert<2\pi$, $\lVert Z\rVert<2\pi$, and $e^{iZ}=e^{iX}e^{iY}$. Then $\tau(Z)=\tau(X)+\tau(Y)$.
\end{lemma}
\begin{proof}
Consider the mapping $g\colon [0,1]\to \mathcal{M}$ defined by $g(t)=e^{iX}e^{itY}$. 
By Lemma \ref{sigma}, we get $\sigma(g(t))\subset \{e^{i\alpha}\mid 0\leq \alpha\leq \lVert X\rVert +\lVert Y\rVert\}$ for every $t\in [0,1]$. 
Fix $\beta\in (\lVert X\rVert +\lVert Y\rVert,2\pi)$.
Let $f\colon \{re^{i\alpha}\mid r>0,\, \beta-2\pi< \alpha<\beta\}\to \mathbb{C}$ be the inverse mapping of the exponential function on $\{a+bi\mid a, b\in \mathbb{R},\ \beta-2\pi< b<\beta\}$.
Observe that $g'(t)=e^{iX}e^{itY}\cdot iY=g(t)\cdot iY$ and $f'(z) = 1/z$.
Thus Lemma \ref{dix} implies that 
\[
\frac{d}{dt}\tau(f\circ g(t))=\tau(g(t)^{-1}g'(t)) = i\tau(Y).
\]
By integrating over $[0,1]$, we get
\[
\tau(f\circ g(1))-\tau(f\circ g(0)) =i\tau(Y).
\]
Since $g(1)=e^{iX}e^{iY}=e^{iZ}$ and $g(0)=e^{iX}$, we obtain $f\circ g(1)=iZ$ and $f\circ g(0)=iX$. Thus the desired conclusion follows.
\end{proof}

Now we are ready to prove one-half of Theorem \ref{tau}.
\begin{proof}[Proof of $\nu_\cM(P) \geq (2\cos \theta)^{-1}$]
Let $A_k$ and $B_k$ be as above.
Assume that 
\[
\max_{1\leq k\leq n}\lVert (P-A_{k-1})^{1/2}A_k^{1/2}\rVert< (2\cos\theta)^{-1}.
\]
By Lemmas \ref{la}, \ref{XY}, \ref{lb}, we have 
$\tau(B_k)-\tau(B_{k-1})< \delta(A_k-A_{k-1})\theta$ for every $k$. 
Thus we obtain $\tau(B_n)-\tau(B_0)< \theta$. 
Since $B_0=0$ and $B_n =(\pi-2\theta)P$, we get $(\pi-2\theta)\tau(P)< \theta$, so $\displaystyle \theta > \frac{\tau(P)\pi}{1+2\tau(P)}$, a contradiction.
\end{proof}

We continue with the other half of the proof. 
We first study the case where $\cM=\mathbb{M}_n$ is the $n\times n$ matrix algebra and $\tau(X)=\mathrm{tr}\, X/n$, $X\in \mathbb{M}_n$. 

\begin{lemma}\label{XYZ}
Let $X, Y, Z\in \mathbb{M}_n$ be self-adjoint ($=$hermitian) matrices. 
If $X-Z$ is positive and invertible, and if $\mathrm{rank}\, (X-Y)+\mathrm{rank}\, (Y-Z)=n$, then $X\geq Y\geq Z$.
\end{lemma}
\begin{proof}
As in the proof of Lemma \ref{XYZW},  one may assume $Z=0$ and $X=I$ without loss of generality. 
In that case, we have $\mathrm{rank}\, (I-Y)+\mathrm{rank}\, Y=n$, which shows that $Y$ is a projection, so $I\geq Y\geq 0$, as desired. 
\end{proof}

\begin{proposition}\label{thetak}
Let $0\neq P\in \cP(\mathbb{M}_n)$ and $\displaystyle \theta := \frac{\tau(P)\pi}{1+2\tau(P)}$.
Let $0=A_0\leq A_1\leq \cdots\leq A_n=P$ be in $\mathbb{M}_n$ satisfying $\mathrm{rank}(A_k-A_{k-1})= 1$ for every $k\in \{1,2,\dots,n\}$. 
Take $B_k\in \mathbb{M}_n$ satisfying $0\leq B_k\leq (\pi-2\theta) P$ and the equation \eqref{akbk}.
Then $\displaystyle \max_{1\leq k\leq n}\lVert (P-A_{k-1})^{1/2}A_k^{1/2}\rVert\leq (2\cos \theta)^{-1}$ holds if and only if $\mathrm{tr}\, B_k=k \theta$ holds for every $k\in \{1,2,\dots,n\}$.
\end{proposition}
\begin{proof}
Assume that $\displaystyle \max_{1\leq k\leq n}\lVert (P-A_{k-1})^{1/2}A_k^{1/2}\rVert\leq (2\cos \theta)^{-1}$ holds. 
By Lemma \ref{lb}, we have $\tau(B_k)-\tau(B_{k-1})\leq\theta/n$, or equivalently, $\mathrm{tr}\, (B_k-B_{k-1})\leq \theta$ for each $k$. This together with $\mathrm{tr}\, B_0=0$ and $\mathrm{tr}\, B_n=\mathrm{tr}\, ((\pi-2\theta)P) =n\theta$ shows that $\mathrm{tr}\, B_k=k\theta$.

Conversely, assume that $\mathrm{tr}\, B_k=k\theta$ holds for every $k\in \{1,2,\dots,n\}$.
We show $\displaystyle \max_{1\leq k\leq n}\lVert (P-A_{k-1})^{1/2}A_k^{1/2}\rVert\leq (2\cos \theta)^{-1}$. 
By Lemma \ref{XY}, it suffices to show that $\cot (B_k+{\varphi}I)\geq \cot (B_{k-1}+{3\varphi}I)$ for every $k\in \{1,2,\dots, n\}$. 
Let $k\in \{1,2,\dots, n\}$. 
By Lemma \ref{la}, we know that 
\begin{equation}\label{1}
\mathrm{rank}\,(\cot (B_{k-1}+{\varphi}I)-\cot (B_k+{\varphi}I))=\mathrm{rank}\,(A_{k-1}-A_k)=1.
\end{equation}
By imitating the calculation in \eqref{cote}, we get 
\[
\begin{split}
&\quad \cot (B_{k-1}+{\varphi}I)-\cot (B_k+{\varphi}I)\\
&= 2i(e^{2i(B_{k-1}+\varphi I)}-I)^{-1}(e^{2i(B_k+\varphi I)}-e^{2i(B_{k-1}+\varphi I)})(e^{2i(B_k+\varphi I)}-I)^{-1}\\
&= 2ie^{2i\varphi}(e^{2i(B_{k-1}+\varphi I)}-I)^{-1}(e^{2iB_k}-e^{2iB_{k-1}})(e^{2i(B_k+\varphi I)}-I)^{-1}, 
\end{split}
\]
which implies $\mathrm{rank}\,(e^{2iB_k}-e^{2iB_{k-1}})=1$ and thus $\mathrm{rank}\,(e^{-2iB_{k-1}}e^{2iB_k}-I)=1$. 
On the other hand, we know $\mathrm{tr}\, (B_k-B_{k-1})=\theta$ from the assumption. 
Recall that $\det e^{X}=e^{\mathrm{tr}\, X}$ holds for every $X\in \mathbb{M}_n$.
% (this can be verified via triangulation).
Thus we get 
\[
\det(e^{-2iB_{k-1}}e^{2iB_k}) =e^{\mathrm{tr}\, (-2iB_{k-1})}\cdot e^{\mathrm{tr}\, (2iB_{k})}  
= e^{2i\theta}.
\] 
It follows that there is a projection $Q\in \mathbb{M}_n$ of rank one satisfying $e^{-2iB_{k-1}}e^{2iB_k}=e^{2i\theta Q}$. 
Then
\[
\mathrm{rank}\, (e^{2iB_k}-e^{2i(B_{k-1}+\theta I)})= \mathrm{rank}\,(e^{2iB_{k-1}}(e^{2i\theta Q}-e^{2i\theta} I))=n-1. 
\]
By imitating the calculation in \eqref{cote}, we also get 
\[
\begin{split}
&\quad \cot (B_{k-1}+{3\varphi}I)-\cot (B_k+{\varphi}I)\\
&= 2i(e^{2i(B_{k-1}+3\varphi I)}-I)^{-1}(e^{2i(B_k+\varphi I)}-e^{2i(B_{k-1}+3\varphi I)})(e^{2i(B_k+\varphi I)}-I)^{-1}\\
&= 2ie^{2i\varphi}(e^{2i(B_{k-1}+3\varphi I)}-I)^{-1}(e^{2iB_k}-e^{2i(B_{k-1}+\theta I)})(e^{2i(B_k+\varphi I)}-I)^{-1}.
\end{split}
\]
Thus we get  
\begin{equation}\label{n-1}
\mathrm{rank}\, (\cot (B_{k-1}+{3\varphi}I)-\cot (B_k+{\varphi}I))=n-1.
\end{equation} 
Since $\cot (B_{k-1}+{\varphi}I)- \cot (B_{k-1}+{3\varphi}I)$ is positive and invertible, Lemma \ref{XYZ} with \eqref{1} and \eqref{n-1} shows that  
\[
\cot (B_{k-1}+{\varphi}I)\geq \cot (B_k+{\varphi}I)\geq \cot (B_{k-1}+{3\varphi}I), 
\]
which completes the proof.
\end{proof}

We will need the following well-known fact in matrix analysis. 
\begin{lemma}[See for example {\cite[Theorem 4.3.26]{HJ}}]\label{matrix}
Let 
\[
\alpha_1'\geq\alpha_1\geq \alpha_2'\geq\alpha_2\geq \cdots\geq \alpha_n'\geq\alpha_n
\]
be real numbers. 
Let $A\in \mathbb{M}_n$ be a self-adjoint matrix whose eigenvalues are $\alpha_1\geq\alpha_2\geq \cdots\geq \alpha_n$. 
Then there is a self-adjoint matrix $A'\in \mathbb{M}_n$ whose eigenvalues are $\alpha_1'\geq\alpha_2'\geq \cdots\geq \alpha_n'$ with $A\leq A'$ and $\mathrm{rank}\, (A'-A)\leq 1$.
\end{lemma}

\begin{lemma}\label{findim}
Let $0\neq P\in \cP(\mathbb{M}_n)$ and $\displaystyle \theta := \frac{\tau(P)\pi}{1+2\tau(P)}$.
Then $\nu_{\mathbb{M}_n}(P) \leq (2\cos \theta)^{-1}$.
\end{lemma}
\begin{proof}
Fix real numbers $\beta_{kl}\in [0,\pi-2\theta]$ ($0\leq k\leq n$, $1\leq l\leq n$) satisfying the following conditions:
\begin{enumerate}
\item[(a)] $\displaystyle \sum_{l=1}^n\beta_{kl}=k\theta$ for every $k\in \{0,1,\dots, n\}$. 
\item[(b)] $\beta_{k,1}\geq\beta_{k-1, 1}\geq \beta_{k,2}\geq \beta_{k-1, 2}\geq \cdots\geq \beta_{k,n}\geq \beta_{k-1,n}$ for every $k\in \{1,2,\dots, n\}$. 
\item[(c)] $\beta_{nl}=\pi-2\theta$ for every $l\leq \mathrm{rank}\, P$, and $\beta_{nl}=0$ for every $l> \mathrm{rank}\, P$.
\end{enumerate}
For example, we may choose $\beta_{kl}=\max\{\min\{\pi-2\theta, k\theta -(l-1)(\pi-2\theta)\}, 0\}$.

We set 
\[
\alpha_{kl}= \frac{1}{2\cos\theta} \cdot \frac{\sin \beta_{kl}}{\sin (\beta_{kl} +\theta)}\in [0,1].
\]
Since the mapping $t\mapsto \sin t/\sin (t +\theta)$ is monotone on $[0,\pi-2\theta]$, we see from (b) that 
\[
\alpha_{k,1}\geq\alpha_{k-1, 1}\geq \alpha_{k,2}\geq \alpha_{k-1, 2}\geq \cdots\geq \alpha_{k,n}\geq \alpha_{k-1,n}
\]
for every $k\in \{1,2,\dots, n\}$.
Therefore, Lemma \ref{matrix} implies that there are $0=A_0\leq A_1\leq \cdots\leq A_n$ in $\mathbb{M}_n$ such that $\mathrm{rank} (A_k-A_{k-1})=1$ and
$\alpha_{k1}\geq \alpha_{k2}\geq \cdots\geq \alpha_{kn}$ are the eigenvalues of $A_k$ for every $k\in \{1,2,\dots, n\}$. 
By (c), $A_n$ is a projection that has the same rank as $P$. 
Thus $A_n$ is unitarily equivalent to $P$, which implies $\nu_{\mathbb{M}_n}(A_n)=\nu_{\mathbb{M}_n}(P)$. 
Take $B_k\in \mathbb{M}_n$ satisfying $0\leq B_k\leq (\pi-2\theta) A_n$ and the equation \eqref{akbk}. Then (a) implies that $\mathrm{tr}\, B_k=k\theta$ for every $k\in \{1,2,\dots,n\}$. 
Therefore, by Propositions \ref{equivalent} and \ref{thetak}, we obtain $\nu_{\mathbb{M}_n}(A_n)\leq (2\cos \theta)^{-1}$.
\end{proof}

Let us get back to the general situation and complete the proof of Theorem \ref{tau}. 
\begin{proof}[Proof of $\nu_\cM(P) \leq (2\cos \theta)^{-1}$]
If $\cM$ is of type I, then Lemma \ref{findim} gives the desired conclusion. 
Assume that $\cM$ is of type II$_1$.
Let $n\geq 1$ be an integer and take the unique integer $a_n\in \{1,2,\dots, n\}$ with $(a_n-1)/n< \tau(P)\leq a_n/n$.
Since $\cM$ is a type II$_1$ factor, one may take a projection $Q\in \cP(\cM)$ satisfying $P\leq Q$ and $\tau(P)=\tau(Q)\cdot a_n/n$.
One may also take a system of matrix units $\{E_{ij}\}_{1\leq i,j\leq n}$ in the algebra $QMQ$ with $E_{11}+\cdots+E_{a_n a_n}=P$.
Let $\cM_1$ be the algebra generated by $\{E_{ij}\}_{1\leq i,j\leq n}$.
Observe that $P\in \cM_1$ and that $\cM_1$ can be identified with $\mathbb{M}_n$. 
Set $\displaystyle \theta_n := \frac{(a_n/n)\pi}{1+2(a_n/n)}$.
By Lemma \ref{findim}, we obtain $\nu_{\cM_1}(P) \leq (2\cos \theta_n)^{-1}$. 
It is clear that $\nu_{\cM}(P)\leq \nu_{\cM_1}(P)$.
Since $\theta_n\to \theta$ as $n\to \infty$, we get $\nu_\cM(P) \leq (2\cos \theta)^{-1}$, as desired.
\end{proof}

\end{document}